\documentclass[12pt]{article}

\evensidemargin
\oddsidemargin
\usepackage{mathrsfs}
\usepackage[margin=2cm]{geometry}
\usepackage{amsmath,amssymb,amsthm,amsfonts}
\usepackage{indentfirst}
\usepackage{bbold}
\usepackage{enumerate}
\usepackage{url}

\makeatletter

\renewcommand\section{\@startsection
{section}
{1}
{\z@}
{0.5\baselineskip}
{0.5\baselineskip}
{\centering\normalfont}}

\newcommand{\Kl}{\left(}
\newcommand{\Kr}{\right)}

\newcommand{\hh}{\hspace{4pt}}
\newcommand{\hhh}{\hspace{8pt}}
\newcommand{\hhhh}{\hspace{16pt}}
\newcommand{\NN}{\mathbb{N}}
\newcommand{\RR}{\mathbb{R}}

\newcommand{\Ff}{\mathcal{F}}

\newcommand{\Pc}{\mathcal{P}}

\newcommand{\Ptwo}{\mathcal{P}_2}
\renewcommand{\div}{\mathop{\mathrm{div}}}
\newcommand{\grad}{\mathop{\mathrm{grad}}}

\newcommand{\abs}[1]{\left\lvert#1\right\rvert}
\newcommand{\Prob}{\mathop{\mathrm{Prob}}}

\newcommand{\AC}{\mathop{\mathrm{AC}}\!}

\newtheorem{theorem}{Theorem}[section]
\newtheorem{lemma}[theorem]{Lemma}
\newtheorem{proposition}[theorem]{Proposition}
\newtheorem{corollary}[theorem]{Corollary}
\newtheorem{remark}[theorem]{Remark}
\newtheorem{assumption}[theorem]{Assumption}

\begin{document}

\title{Wasserstein gradient flows from large deviations of thermodynamic limits}
\author{Manh Hong Duong\textsuperscript{1}, Vaios Laschos\textsuperscript{1} and Michiel Renger\textsuperscript{2}
\\ \small{\textsuperscript{1} Department of Mathematical sciences, University of Bath, }
\\ \small{\textsuperscript{2} ICMS and Dep. of Math. and Comp. Sciences, TU Eindhoven.}} 

\title{Wasserstein gradient flows from large deviations of thermodynamic limits}

\date{\today}

\maketitle
\begin{abstract}

We study the Fokker-Planck equation as the thermodynamic limit of a stochastic particle system on one hand and as a Wasserstein gradient flow on the other.
We write the rate functional, which characterizes the large deviations from the thermodynamic limit, in such a way that the free energy appears explicitly.
Next we use this formulation via the contraction principle to prove that the discrete time rate functional is asymptotically equivalent in the Gamma-convergence sense to the functional derived from the Wasserstein gradient discretization scheme. 
\end{abstract}

\maketitle
\section{Introduction}
\label{sec:introduction}
Since the seminal work of Jordan, Otto and Kinderlehrer \cite{Jordan1998}, it has become clear that there are many more partial differential equations that can be written as a gradient flow than previously known. Two important insights have contributed to this: the generalisation of gradient flows to metric spaces and the specific choice of the Wasserstein metric as the dissipation mechanism. The paper by Jordan, Kinderlehrer and Otto introduced a gradient-flow structure by approximation in discrete time. More recent work have shown how these ideas can be studied in continuous time \cite{Otto2001}, and how they can be generalised to any metric space \cite{Ambrosio2008}. This paper is mainly concerned with the time-discrete scheme, which we shall now explain.

A gradient flow in $L^2(\RR^d)$ is an evolution equation of the form
\begin{equation}
\label{def:L2 gradient flow}
  \frac{\partial \rho}{\partial t}=-\text{grad}_{L^2} \mathcal{F}(\rho),
\end{equation}
for some functional $\mathcal{F}$. For a gradient flow it is natural to use the following time-discrete variational scheme. If $\rho_0$ is the solution at time $t=0$, then the solution at time $\tau>0$ is approximated by the minimiser of the functional
\begin{equation*}
\label{def:L2 variational scheme}
  \rho\mapsto \mathcal{F}(\rho) + \frac1{2\tau}\|\rho-\rho_0\|^2_{L^2({\RR^d})}.
\end{equation*}
Indeed, the Euler-Lagrange equation is then $\frac{\rho_\tau-\rho_0}{\tau} = -\text{grad}_{L^2} \mathcal{F}(\rho_\tau)$, which clearly approximates \eqref{def:L2 gradient flow} as $\tau\to0$. In the same manner, one can define a variational scheme by minimising the functional
\begin{equation}
\label{def:W2 variational scheme}
  \rho\mapsto \mathcal{F}(\rho) + \frac1{2\tau}W_2^2(\rho,\rho_0),
\end{equation}
where $W_2$ is the Wasserstein metric. Convergence of this variational scheme was first proven in \cite{Jordan1998} with the choice of $\mathcal{F}(\rho):=\mathcal{S}(\rho)+\mathcal{E}(\rho)$, where
\begin{align}
\label{eq:free energy}
  &&\mathcal{E}(\rho)=\int_{\RR^{d}}\Psi(x)\rho(dx)
  &&\text{and}
  &&\mathcal{S}(\rho):= \begin{cases}\int_{\RR^{d}}\rho(x)\log\rho(x)\,dx,  &\text{for } \rho(dx)=\rho(x)\,dx\\
                                       \infty,                                &\text{otherwise},
                          \end{cases}
\end{align}
for some potential $\Psi$. In this case, the minimisers converge to the solution of the Fokker-Planck equation
\begin{equation}
\label{eq:Fokker-Planck}
  \frac{\partial \rho}{\partial t} = \Delta \rho + \div (\rho \nabla\Psi).
\end{equation}
Later, in \cite{Otto2001}, this result was extended to more general $\mathcal{F}$, but we will be concerned with the specific choice~\eqref{eq:free energy}. Physically, $\mathcal{S}$ can be interpreted as entropy, $\mathcal{E}$ as internal energy, and $\mathcal{F}$ as the corresponding Helmholtz free energy (if the temperature effects are hidden in $\Psi$); hence it is not surprising that this free energy should decay along solutions of \eqref{eq:Fokker-Planck}. However, it is not intuitively clear why the dissipation of free energy must be described by the Wasserstein metric.

As we will explain in Section~\ref{sec:part system ldps}, for systems in equilibrium, the stochastic fluctuations around the equilibrium are characterised by a free energy similar to~\eqref{eq:free energy}. Recent developments suggest a similar principle for systems away from equilibrium \cite{Leonard2007, Adams2010, Peletier2011, Laschos2012, Adams2012}. To explain this, consider $N$ independent random particles in $\RR^d$ with positions $X_k(t)$, initially distributed by some $\rho_0\in\mathcal{P}(\RR^d)$, where the probability distribution of each particle evolves according to \eqref{eq:Fokker-Planck}. Define the corresponding \emph{empirical process}
\begin{equation*}
  L_N:t \mapsto \frac{1}{N}\sum_{k=1}^{N}\delta_{X_k(t)}.
\end{equation*}
Then, as a consequence of the Law of Large Numbers, at each $\tau\geq0$ the \emph{empirical measure} $L_N(\tau)$ converges almost surely in the narrow topology as $N\to\infty$ to the solution of the Fokker-Planck equation \eqref{eq:Fokker-Planck} with initial condition $\rho_0$ \cite{Dudley1989}; this is sometimes known as the thermodynamic limit. The rate of this convergence is characterised by a large deviation principle. Roughly speaking, this means that there exists a $J_\tau:\Pc(\RR^d)\to[0,\infty]$ such that (see Section~\ref{sec:part system ldps})
\begin{equation*}
  \Prob\left(L_N(\tau)\approx \rho \,\vert\, L_N(0)\approx \rho_0\right) \sim \exp\left(-N\,J_\tau(\rho|\rho_0)\right) \hspace{1cm} \text{ as } N\to\infty.
\end{equation*}

In \cite[Prop. 3.2]{Leonard2007} and \cite[Cor. 13]{Peletier2011}, it was found that
\begin{equation}
\label{eq:quenched ldp}
  J_\tau(\rho|\rho_0)=\inf\Big\{ \mathcal{H}(\gamma|\rho_0\otimes p_\tau): \gamma\in\Pi(\rho_{0},\rho)\Big\},
\end{equation}
where $\mathcal{H}$ is the relative entropy (discussed in Section~\ref{sec:part system ldps}), $p_t$ is the fundamental solution of the Fokker-Planck equation~\eqref{eq:Fokker-Planck} and $\Pi(\rho_{0},\rho)$ is the set of all Borel measures in $\RR^{2d}$ that have first and second marginal $\rho_{0}$ and $\rho$ respectively. In this paper, we characterise a class of potentials $\Psi$ and initial data $\rho_{0}$ for which~\eqref{eq:quenched ldp} is equal to
\begin{equation}
\label{eq:contracted path ldp}
  J_\tau(\rho|\rho_0)=\inf_{\rho_{(\cdot)}\in C_{W_{2}}(\rho_0,\rho)}\Bigg\{\frac{1}{4\tau}\int_{0}^{1}\left\|\frac{\partial\rho_{t}}{\partial t}\right\|^{2}_{-1,\rho_{t}}dt + \frac{\tau}{4}\int_{0}^{1}\left\|\grad\Ff(\rho_{t})\right\|^{2}_{-1,\rho_{t}}\,dt +\frac{1}{2}\mathcal{F}(\rho_{1})-\frac{1}{2}\mathcal{F}(\rho_{0})\Bigg\}.
\end{equation}
where the $\| \cdot\|_{-1,\rho}$ norm and the exact meaning of $ \grad \Ff$ will be defined in the sequel.
 In the main theorem, by using the above equality, we show that the Wasserstein scheme \cite{Jordan1998} has the same asymptotic behavior with $J_\tau$ for $\tau\rightarrow 0$, in terms of Gamma-convergence (see \cite{Braides2002} for an exposition of Gamma-convergence).
\begin{theorem}\label{th:main result}Let  $\rho_{0}=\rho_0(x)dx\in\Pc_2(\RR)$  be absolutely continuous with respect to the Lesbegue measure with $\rho_0(x)$  is bounded from below by a positive constant in every compact set. Assume that  $\mathcal{F}(\rho_{0}),\|\Delta\rho_{0}\|^{2}_{-1,\rho_{0}}$ and $\int_{\RR}|\nabla\Psi(x)|^{2}\,\rho_0(dx)$ are all finite, and that $\Psi\in C^{2}(\RR)$ satisfies either Assumption~\ref{as:subquadratic} or ~\ref{as:superquadratic} (introduced in Section~\ref{sec:large deviations}).  Then we have
\begin{equation}
\label{eq:main result}
  J_{\tau}(\cdot\,|\rho_0)-\frac{W_{2}^{2}(\rho_{0},\,\cdot\,)}{4\tau}\xrightarrow[\tau\to0]{\Gamma} \frac{1}{2}\mathcal{F}(\,\cdot\,)-\frac{1}{2}\mathcal{F}(\rho_{0}), \hspace{3cm}\text{in } \Ptwo(\RR).
\end{equation}
\end{theorem}
Here $\Ptwo(\RR)$ denotes the space of probability measures on $\RR$ having finite second moment. As we will prove, the Gamma-convergence result holds if $\Ptwo(\RR)$ is equipped with the narrow topology, as well as if we equip it with the Wasserstein topology. More precisely: we will prove the lower bound in the narrow topology (Theorem~\ref{th:lower bound}), and the existence of the recovery sequence (Theorem~\ref{th:recovery sequence}) in the Wasserstein topology. In the Wasserstein topology, the Gamma-convergence~\eqref{eq:main result} immediately implies:
\begin{equation}
\label{eq:1st order gamma convergence}
  \tau\,J_\tau(\cdot\,|\rho_{0}) \xrightarrow[\tau\to0]{\Gamma} \frac14 W_2^2(\rho_0,\,\cdot\,)  \hspace{1.5cm} \text{ in } \Pc_2(\RR).
\end{equation}
For a system of Brownian particles, i.e. $\Psi\equiv0$, statement~\eqref{eq:1st order gamma convergence} can also be found in \cite{Leonard2007}. Together, the two statements \eqref{eq:main result} and \eqref{eq:1st order gamma convergence} make up an asymptotic development of the rate $J_\tau$ for small $\tau$, i.e.
\begin{equation*}
  J_\tau(\rho|\rho_{0}) \approx \frac12 \mathcal{F}(\rho)-\frac12 \mathcal{F}(\rho_0) + \frac1{4\tau} W_2^2(\rho_0,\rho).
\end{equation*}
Apart from the factor $1/2$ and the constant $\mathcal{F}(\rho_0)$, which do not affect the minimisers, this approximation indeed corresponds to the functional defining the time-discrete variational scheme~\eqref{def:W2 variational scheme} from \cite{Jordan1998}.

For $\Psi=0$, the main statement~\eqref{eq:main result} was proven in \cite{Adams2010} in a subset of $\Ptwo(\RR)$ consisting of measures that are sufficiently close to a uniform distribution on a compact interval. In \cite{Peletier2011}, it was proven that whenever \eqref{eq:main result} holds for $\Psi=0$, then it also holds for any $\Psi\in C^2_b(\RR^d)$. Both papers make use of the specific form of the fundamental solution of~\eqref{eq:Fokker-Planck}. In \cite {Laschos2012}, ~\eqref{eq:main result} was shown for Gaussian measures on the real line. In our approach, using the path-wise large deviations, we can avoid using the fundamental solution, allowing us to prove the statement in a much more general context.

All theorems in this paper also work in higher dimensions, except for the existence of the recovery sequence in the main theorem. This has to do with the fact that in one dimension the optimal transport plan between two measures with equal tails will be the identity at the tails. However, this argument fails in higher dimensions. We belief that the recovery sequence also exists in higher dimensions but this is left for future research.

The required concepts of this paper are introduced in Section~\ref{sec:preliminaries}. In Section~\ref{sec:part system ldps}, we explain the concept of large deviations in the case of an equilibrium system, introduce the dynamical particle system that we study more precisely, and discuss the conditional large deviations for this system. The alternative form of the functional $J_h$ is proven in Section~\ref{sec:large deviations} via the path-wise large deviation principles. Finally, in Section~\ref{sec:lower bound} we prove the Gamma-convergence lower bound, and in Section~\ref{sec:recovery sequence} the existence of the recovery sequence.

\section{Preliminaries}
\label{sec:preliminaries}

By the nature of this study, we need a combination of techniques from probability theory, mostly from the theory of large deviations, and from functional analysis, mostly from the gradient flow calculus as set out in \cite{Ambrosio2008}. Let us introduce these concepts here.

To begin, let us discuss the topological measure spaces. Unless otherwise stated, the space of probability measures $\Pc(\RR^d)$ will be endowed with the narrow topology, defined by convergence against continuous bounded test functions:
\begin{align*}
  \rho_t \to \rho \text{ as } t\to0 \text{ if and only if } \int_{\RR^d}\!\phi\,d\rho_t \to \int_{\RR^d}\!\phi\,d\rho \text{ for all } \phi\in C_b(\RR^d).
\end{align*}
We sometimes identify measures with densities when possible, which is typically the case if a measure has finite entropy. The space $\Ptwo(\RR^d)=\left\{ \rho \in \Pc(\RR^d): \int\! |x|^2\,\rho(dx) < \infty \right\}$ will be endowed with the topology generated by the Wasserstein metric $W_2$.  The Wasserstein distance of two measures $\rho_{0},\rho \in \Ptwo(\RR^{d})$ is defined via
\begin{equation*}
  W_{2}^{2}(\rho_{0},\rho)=
  \inf_{\gamma\in\Pi(\rho_{0},\rho)}\left\{\int_{\RR^{n}}\int_{\RR^{n}}
  |x-y|^{2} d\gamma\right\}.
\end{equation*}

Convergence in the Wasserstein topology can be characterised as (see e.g. \cite{Villani2003, Ambrosio2008}):
\begin{align*}
  \rho_t \to \rho \text{ as } t\to0 \text{ if and only if } &(i) \quad \rho_t\to\rho \text{ narrowly, and} \\
                                                            &(ii) \quad \int_{\RR^d}\!|x|^2\,d\rho_t\to\int_{\RR^d}\!|x|^2\,d\rho.
\end{align*}

We write $C([0,1],\Pc(\RR^{d}))$ for the space of narrowly continuous curves $[0,1]\to\Pc(\RR^d)$, and $C(\rho_0,\rho)$ for the space of narrowly continuous curves $[0,1]\to\Pc(\RR^{d})$ starting in $\rho_0$ and ending in $\rho$. Similarly, for Wasserstein-continuous curves in $\Ptwo(\RR^d)$ we write $C_{W_2}([0,1],\Ptwo(\RR^{d}))$ and $C_{W_2}(\rho_0,\rho)$.

Furthermore, we use two different notions of absolutely continuous curves. The first notion is taken from \cite[Def. 4.1]{Dawson1987}. Let $\mathcal{D}=C_c^\infty(\RR^d)$ be the space of test functions with the corresponding topology (see \cite[Sect. 6.3]{Rudin1973}), let $\mathcal{D}'$ be its dual, consisting of the associated distributions, and let $\langle\,,\,\rangle$ be the dual pairing between $\mathcal{D}'$ and $\mathcal{D}$. We will identify a measure $\rho\in \Pc(\RR^d)$ with a distribution by setting $\langle\rho,f\rangle:=\int\!f \,d\rho$. Denote by $\mathcal{D}_K\subset \mathcal{D}$ the subspace of all Schwartz functions with compact support $K\subset\RR^d$. Then a curve $\rho_{(\cdot)}: [0,1]\rightarrow\mathcal{D'}$ is said to \textit{be absolutely continuous in the distributional sense} if for each compact set $K\subset \RR^{d}$ there is a neighborhood $U_K$ of $0$ in $\mathcal{D}_K$ and an absolutely continuous function $G_K: [0,1]\rightarrow\RR$ such that
\begin{equation*}
  |\langle\rho_{t_2},f\rangle-\langle\rho_{t_1},f\rangle|\leq |G_K(t_2)-G_K(t_1)|,
\end{equation*}
for all $0<t_1,t_2<1$ and $f\in U_K$. We denote by $\AC\Kl[0,1];\mathcal{D}'\Kr$ the set of all absolutely continuous maps in distributional sense. Note that if a map $\rho_{(\cdot)}: [0,1]\rightarrow\mathcal{D'}$ is absolutely continuous then the derivative in the distributional sense $\dot{\rho}_t=\lim_{\tau\rightarrow0}\frac{1}{\tau}(\rho_{t+\tau}-\rho_{t})$ exists for almost all $t\in [0,1]$.

Secondly, we say a curve $\rho_{(\cdot)}:[0,1]\to\Ptwo(\RR^d)$ is \textit{absolutely continuous in the Wasserstein sense} if there exists a $g\in L^1(0,1)$ such that
\begin{equation*}
  W_2(\rho_{t_1},\rho_{t_2}) \leq \int_{t_1}^{t_2} g(t)\,dt
\end{equation*}
for all $0<t_1 \leq t_2<1$ (see for example \cite{Ambrosio2008}). We denote the set of absolutely continuous curves as $\AC_{W_2}([0,1];\Ptwo(\RR^d))$.

For an absolutely continuous curve $\rho_{(\cdot)}$ there is a unique Borel field $v_t \in V:=\overline{\left\{\nabla p:p \in \mathcal{D}\right\}}^{L^2(\rho_t)}$ such that the continuity equation holds \cite[Th. 8.3.1]{Ambrosio2008}:
\begin{equation}
\label{eq:cont equation}
  \frac{\partial\rho_t}{\partial t}+ \div(\rho_t \,v_t)=0 \hspace{1cm} \text{in distributional sense.}
\end{equation}
This motivates the identification of the tangent space\footnote{Here we like to point out that in \cite{Ambrosio2008} the tangent space is identified with the set of velocity fields $V$.} of $\Ptwo(\RR^d)$ at $\rho$ with all $s\in \mathcal{D}'$ for which there exists a $v\in V$ such that
\begin{equation}
\label{eq:tangent space condition}
  s + \div(\rho \,v)=0 \hspace{1cm} \text{in distributional sense}.
\end{equation}
The following inner product on the tangent space at $\rho$ is the metric tensor corresponding to the Wasserstein metric \cite{Otto2001}
\begin{equation*}
  ( s_1, s_2)_{-1,\rho}:=\frac{1}{2}\int_{\RR^d}\!v_1 \cdot v_2\,d\rho,
\end{equation*}
where $v_1$ and $v_2$ are associated with $s_1$ and $s_2$ through \eqref{eq:tangent space condition}. The corresponding norm coincides with the dual operator norm on $\mathcal{D'}$
\begin{equation}
\label{def:-1 norm}
  \| s \|^2_{-1,\rho} := \sup_{p\in\mathcal{D}} \,\left\{ \langle s,p\rangle - \frac12\int_{\RR^d}\!|\nabla p|^2d\rho\right\}.
\end{equation}
This norm is closely related to the Wasserstein metric through the Benamou-Brenier formula \cite{Benamou2000}
\begin{equation}
\label{eq:Benamou-Brenier}
  W_2(\rho_0,\rho_1)^2=\min\left\{ \int_0^1 \! \|\frac{\partial\rho_t}{\partial t} \|^2_{-1,\rho_t}\,dt: \rho_t|_{t=0}=\rho_0 \text{ and } \rho_t|_{t=1}=\rho_1\right\}.
\end{equation}

Observe that, in approximation, any small perturbation $\rho_t$ from a $\rho\in \Ptwo(\RR^d)$ can be specified by a potential $p \in \mathcal{D}$ such that \eqref{eq:cont equation} holds with $\rho_0=\rho$ and $v=\nabla p$. Following \cite[Definition 9.36]{Feng2006}, for any $\Ff\colon\mathcal{P}(\RR^{d})\rightarrow [-\infty,+\infty]$, we write, if it exists, $\grad\Ff(\rho)$ for the unique element in $\mathcal{D}'$ such that for each $p \in \mathcal{D}$ and each $\rho_{(\cdot)}\colon [0,\infty)\rightarrow\mathcal{P}(\RR^{d})$ satisfying \eqref{eq:cont equation} with $\rho_0=\rho$ and $v=\nabla p$, we have
\begin{equation*}
  \lim_{t\rightarrow0^+}\frac{\Ff(\rho_{t})-\Ff(\rho)}{t}=\langle\grad\Ff(\rho),p\rangle.
\end{equation*}
Let $\Ff(\rho)=\mathcal{E}(\rho)+\mathcal{S}(\rho)$ be the free energy defined as in \eqref{eq:free energy}. By \cite[Theorem D.28]{Feng2006}, if $\Ff(\rho)<\infty$, then
\begin{equation*}
\grad\Ff(\rho)=-(\Delta\rho+\div(\rho\nabla\Psi)) ~~\text{in}~~ \mathcal{D}'(\RR^d).
\end{equation*}

The following functional plays a central role in this paper
\begin{equation}
\label{FisherI}
\|\Delta \rho\|^{2}_{-1,\rho} = \begin{cases}
\int_{\RR^{d}}\!\tfrac{ |\nabla\rho(x)|^{2}}{\rho(x)}\,dx             &\mbox{if}\hhh  \rho(dx)=\rho(x)\,dx \text{ and } \sqrt{\rho}\in H^{1}(\RR^d),       \\
\infty       &\text{otherwise},
       						\end{cases}
\end{equation}
where $\nabla\rho$ is the distributional derivative of $\rho$. This functional is also known as the \emph{Fisher information}.

We conclude this section with two results that we will need.
\begin{lemma}\cite[Th. 8.3.1]{Ambrosio2008}\label{ACcurve}
Let $\rho_{(\cdot)}\colon(0,\tau)\rightarrow \Pc(\RR^d)$ be a narrowly continuous curve and let $v_{(\cdot)}:(0,\tau)\to V$ be a vector field such that the continuity equation \eqref{eq:cont equation} holds. If
\begin{equation}
  \rho_0\in \Pc_2(\RR^d) ~\text{and}~ \int_0^\tau\big\|v_{t}\big\|^{2}_{L^2(\rho_{t})}dt <\infty\label{AGStheorem}
\end{equation}
then $\rho_t\in \Pc_2(\RR^d)$ for all $0<t<\tau$ and $\rho_{(\cdot)}$ is absolutely continuous in the Wasserstein sense.
\end{lemma}
\begin{remark}
We point out that the hypothesis in Lemma~\ref{ACcurve} requires a priori that the curve $\rho_{(\cdot)}$ lies in $\Pc_2(\RR^d)$, but the proof actually shows that the condition \eqref{AGStheorem} implies the whole curve to be in $\Pc_2(\RR^d)$ (and it is absolutely continuous in the Wasserstein sense).
\end{remark}

\begin{lemma}\label{chainrule}
Assume that $\rho_{(\cdot)}\colon (0,\tau)\to\Pc_2(\RR^d)$ is a Wasserstein-absolutely continuous curve.
\begin{enumerate}
\item If $\Psi\in C^2(\RR^d)$ is convex, bounded from below, and it satisfies the conditions
\begin{equation*}
  \mathcal{E}(\rho_{t})< \infty \hspace{2pt}\quad \forall t\in [0,\tau]\hspace{2pt} \text{ and } \int_0^\tau\int_{\RR^d}|\nabla \Psi(x)|^2\rho_t(x)\,dx\,dt<+\infty,
\end{equation*}
then $t\mapsto\mathcal{E}(\rho_t)$ is absolutely continuous.
\item If
\begin{equation*}
  \mathcal{S}(\rho_{t})< \infty\hspace{2pt}\quad \forall t\in [0,\tau]\hspace{2pt} \text{ and } \int_0^\tau\|\Delta \rho_{t}\|^2_{-1,\rho_{t}}\,dt<\infty,
\end{equation*}
then $t\mapsto\mathcal{S}(\rho_t)$ is absolutely continuous.

If the conditions in both parts are satisfied, then $\grad\Ff(\rho_t)$ exists and the following chain rule holds
\begin{equation}
  \frac{d}{dt} \Ff(\rho_t) = \left( \grad\Ff(\rho_t), \frac{\partial}{\partial t} \rho_t \right)_{-1,\rho_t}.\label{eq:chainrule}
\end{equation}
\end{enumerate}
\end{lemma}
\begin{proof}
This Lemma is a direct consequence of \cite[Th. 10.3.18]{Ambrosio2008}. Since the functionals $\mathcal{E}(\rho)$ and $\mathcal{S}(\rho)$ are lower semicontinuous and geodesically convex, we only need to check condition \cite[10.1.17]{Ambrosio2008}. This condition in turn is satisfied by the Cauchy-Schwartz inequality $\langle f,g\rangle_{L^2(a,b)} \leq \|f\|_{{L^2(a,b)}}\|g\|_{L^2(a,b)}$ and the assumptions.
\end{proof}

\section{Particle system and conditional large deviations}
\label{sec:part system ldps}

In this section we first explain the concept of large deviations with a simple model particle system. Then, we introduce the dynamic particle system that we study more precisely, and discuss the large deviation principle for this system.

Consider a system of independent random particles in $\RR^d$ (without dynamics), where the positions $X_1,\hdots,X_N$ are identically distributed with law $\rho_0$. Then as a consequence of the law of large numbers $L_N\to\rho_0$ almost surely in the narrow topology as $N\to\infty$ \cite[Th. 11.4.1]{Dudley1989}. Naturally, this implies weak convergence:
\begin{equation*}
  \lim_{N\to\infty} \Prob(L_N\in C) = \delta_{\rho_0}(C)
\end{equation*}
for all continuity sets $C\subset\Pc(\RR^d)$ in the narrow topology. A large deviation principle quantifies the exponential rate of convergence to 0 (or 1). More precisely, we say \textit{the system satisfies a large deviation principle in $\Pc(\RR^d)$ with (unique) rate $J:\Pc(\RR^d)\to[0,\infty]$} if $J$ is lower semicontinuous, and for all sets $U\subset \Pc(\RR^d)$ there holds (see, for example \cite{Dembo1998})
\begin{equation*}
  -\inf_{U^\circ} J \leq \liminf_{N\to\infty} \frac1N\log \Prob(L_N\in U^\circ) \leq\limsup_{N\to\infty} \frac1N\log\Prob(L_N\in\overline{U}) \leq -\inf_{\overline{U}} J.
\end{equation*}
In addition, we say a rate functional is \textit{good} if it has compact sub-level sets. By Sanov's Theorem \cite[Th. 6.2.10]{Dembo1998}, our model example indeed satisfies a large deviation principle, where the good rate functional $J(\rho)$ is the relative entropy
\begin{equation}
\label{eq:relative entropy}
  \mathcal{H}(\rho|\rho^0):=\begin{cases} \int\!\log(\frac{d\rho}{d\rho^0})\,d\rho, &\text{if } \rho\ll\rho^0,\\
                                          \infty,                                   &\text{otherwise}.
                            \end{cases}
\end{equation}
In this example we see the (relative) entropy appearing naturally from a limit of a simple particle system.

Let us now consider our particle system with dynamics, and study its Sanov-type large deviations. To define the system more precisely, let $X_1(t),\cdots, X_N(t)$ be a sequence of independent random processes in $\RR^{d}$. Assume that the initial values are fixed deterministically by some $X_1(0)=x_1,\hdots X_N(0)=x_N$ in such a way that \footnote{The reason behind this specific initial condition is that we want to somehow condition on $L_N=\rho$, which is a measure-0 set.}
\begin{equation}
\label{eq:initial condition}
  L_N(0)\to\rho_0 \hspace{1.5cm} \text{narrowly for some given } \rho_0\in\Pc(\RR^{d}).
\end{equation}
The evolution of the system is prescribed by the same transition probability for each particle $\Prob(X_k(t)\in dy|X_k(0)=x)=p_t(dy|x)$. Naturally, for such probability there must hold $p_t(dy|x)\to \delta_x(dy)$ narrowly as $t\to0$, and it should evolve according to \eqref{eq:Fokker-Planck}. We thus define $p_t$ to be the fundamental solution of~\eqref{eq:Fokker-Planck} \footnote{Equivalently, we can define the dynamics of $X_1,\hdots,X_N$ by the It\={o} stochastic equations
\begin{equation*}
\label{model}
  dX_k(t)=-\nabla\Psi(X_k(t))\,dt + \sqrt{2}\,dW_k(t), \quad k=1,\cdots, N
\end{equation*}
where $W_1,\hdots, W_N$ are independent Wiener processes.}.

Again by the law of large numbers $L_N(\tau)\to\rho_\tau$ almost surely in $\Pc(\RR^d)$, where $\rho_\tau=\rho_0\ast p_\tau$, the solution of~\eqref{eq:Fokker-Planck} at time $\tau$ with initial condition $\rho_0$. In addition, the empirical measure $L_N(\tau)$ satisfies a large deviation principle
\begin{equation*}
  \Prob\left(L_N(\tau)\approx\rho\right) \sim \exp\left(-N J_\tau(\rho|\rho_0)\right) \hspace{1.5cm} \text{as } N\to\infty.
\end{equation*}
with good rate functional~\eqref{eq:quenched ldp}. Observe that $J_\tau(\,\cdot\,|\rho_0)\geq0$ is minimised by $\rho_0\ast p_\tau$.

\section{Large deviations of trajectories}
\label{sec:large deviations}
In this section we prove, under suitable assumptions for $\rho_{0}$ and $\Psi$, the equivalence of the rate functionals~\eqref{eq:quenched ldp} and \eqref{eq:contracted path ldp}. The latter form will be used to prove the main Gamma convergence theorem. First, the large deviations of the empirical process is derived. To this aim we will need to distinguish between two different types of potentials $\Psi$. Next, we transform these large deviation principles back to the large deviations of the empirical measure $L_N(\tau)$ by a contraction principle, and finally show that the resulting rate functionals are the same for both cases.

In the first case we consider potentials that satisfy the following
\begin{assumption}[The subquadratic case]\label{as:subquadratic} Let $\Psi \in C^2(\RR^{d})$ such that
  \begin{enumerate}
  \item $\Psi$ is bounded from below,
  \item there is a $C>0$ such that $|x||\nabla \Psi(x)|\leq C(1+|x|^{2})$ for all $x\in\RR^d$,
\item $\Psi$ is convex,
\item $\Delta\Psi$ is bounded.
\end{enumerate}
\end{assumption}
Note that the second assumption indeed implies $|\Psi(x)|\leq C(1+|x|^2)$. Under Assumption~\ref{as:subquadratic}, combined with initial condition~\eqref{eq:initial condition}, the empirical process $\{L_N(t)\}_{0\leq t \leq\tau}$ satisfies a large deviation principle in $C([0,\tau],\Pc(\RR^{d}))$ with good rate functional \cite[Th. 4.5]{Dawson1987}
\begin{equation}
\label{eq:ldp sub and super}
\tilde{J}_\tau(\rho_{(\cdot)})=\begin{cases} \hhh \frac{1}{4}\int_0^\tau\|\frac{\partial \rho_{t}}{\partial t}-\Delta\rho_{t}-\div(\rho_{t}\nabla\Psi)\|^2_{-1,\rho_{t}}dt, & \mbox{if}\hhh \rho_{(\cdot)}\in \AC\Kl[0,\tau];\mathcal{D}'\Kr,\\
                   \infty, &\text{otherwise}.
                 \end{cases}
\end{equation}
It follows from a contraction principle \cite[Th. 4.2.1]{Dembo1998} and a change of variables $t\mapsto t\slash \tau$ that
\begin{equation}\label{eq:contracted ldp subquadratic}
   J_{\tau}(\rho|\rho_0)=\inf_{\rho_{(\cdot)}\in C(\rho_{0},\rho)}  \frac{1}{4\tau}\int_0^1\left\|\frac{\partial \rho_{t}}{\partial t}-\tau(\Delta\rho_{t}+\div(\rho_{t}\nabla\Psi))\right\|^2_{-1,\rho_{t}}\,dt.
\end{equation}

\begin{remark}The first assumption guarantees that the functional $\mathcal{E}:\mathcal{P}(\RR^{d})\to(-\infty,\infty\rbrack$ is well defined.
The last two assumptions are not necessary to derive \eqref{eq:ldp sub and super}; however we will need them in the sequel.
Especially the last one is a technical assumption that we will need in Lemma~\ref{lem:aux1}. It can be relaxed in several
ways, but for simplicity we chose not to.
\end{remark}

\begin{remark} In \eqref{eq:contracted ldp subquadratic} we implicitly set $\frac1{4\tau}\int_0^1\|\tfrac{\partial \rho_{t}}{\partial t}-\tau(\Delta\rho_{t}+\div(\rho_{t}\nabla\Psi))\|^2_{-1,\rho_{t}}\,dt=\infty$ if the curve is not absolutely continuous in distributional sense. Therefore, from now on, we shall only consider curves in $C(\rho_0,\rho)$ or $C_{W_2}(\rho_0,\rho)$ that are absolutely continuous in distributional sense.
\end{remark}

In the second case we require a combination of assumptions on $\Psi$ that were taken from \cite{Feng2006} and \cite{Feng2011}:
\begin{assumption}[The superquadratic case]\label{as:superquadratic} Let $\Psi \in C^4(\RR^{d})$ such that:
  \begin{enumerate}
    \item There is some $\lambda_\Psi \in \RR$ such that $z^T D^2\Psi(x)z\geq \lambda_\Psi\abs{z}^2$ for all $x,z\in\RR^d$;
    \item $\int_{\RR^{d}}\Psi(x) e^{-2\Psi(x)}\,dx < \infty$;
    \item $\Psi$ has superquadratic growth at infinity, i.e. $\lim_{|x|\rightarrow\infty}\frac{\Psi(x)}{|x|^2}=+\infty$;
    \item There exists an $\omega\in C(\RR_+)$ with $\omega(0)=0$ such that for all $x,y\in \RR^d$
      \begin{align*}
        \Psi(y)-\Psi(x)&\leq \omega(|y-x|)(1+\Psi(x)),
      \\|\Psi(y)-\Psi(x)|^2&\leq \omega(|y-x|)(1+|\nabla \Psi(x)|^2+\Psi(x));
      \end{align*}
    \item $\zeta := |\nabla\Psi|^{2}-2\Delta\Psi$ has superquadratic growth at infinity, i.e. $\lim_{|x|\rightarrow\infty}\frac{\zeta(x)}{|x|^2}=+\infty;$
    \item There is some $\lambda_\zeta \in \RR$ such that $z^T D^2\zeta(x)z\geq\lambda_\zeta\abs{z}^2$ for all $x,z\in\RR^d$.
  \end{enumerate}
\end{assumption}
Whenever Assumption~\ref{as:superquadratic} and initial condition \eqref{eq:initial condition} hold, then by \cite[Th. 13.37]{Feng2006} the process $\{L_N(t)\}_{0\leq t \leq \tau}$ satisfies a large deviation principle in $C_{W_{2}}([0,\tau],\Pc_{2}(\RR^{d}))$ with good rate functional~\eqref{eq:ldp sub and super}.

\begin{remark}
Contrary to the subquadratic case, the latter is actually a large deviation principle on the set of all continuous paths in $\Pc_{2}(\RR^{d})$ with respect to the Wasserstein topology. Although we strongly believe that this is also true for the subquadratic case, it is very difficult to prove due to the fact that the functional $\tilde{J}_{\tau}$  does not have Wasserstein-compact sub-level sets, and therefore it can't be a good rate functional in $C_{W_{2}}([0,\tau], \Pc_{2}(\RR^{d}))$ when $\Psi$ is subquadratic.
\end{remark}

Again, by a contraction principle and a simple change of variables, it follows from \eqref{eq:ldp sub and super} that \eqref{eq:quenched ldp} must be equal to:
\begin{equation}\label{eq:contracted ldp superquadratic}
   J_{\tau}(\rho|\rho_0)=\inf_{\rho_{(\cdot)}\in C_{W_2}(\rho_{0},\rho)}  \frac{1}{4\tau}\int_0^1\left\|\frac{\partial \rho_{t}}{\partial t}-\tau(\Delta\rho_{t}+\div(\rho_{t}\nabla\Psi))\right\|^2_{-1,\rho_{t}}\,dt.
\end{equation}

Observe that in this case the infimum is taken over Wasserstein-continuous curves, while in the subquadratic case \eqref{eq:contracted ldp subquadratic} the infimum was over narrowly continuous curves. However, we will prove that under the extra assumption that $\rho_{0}\in\mathcal{P}_{2}(\RR^{d})$ and $\mathcal{F}(\rho_{0})$ is finite, even in the subquadratic case the infimum can be taken over $ C_{W_2}(\rho_{0},\rho) $. Actually, we will prove something even stronger, that we will need in the sequel, namely the following:
\begin{proposition}\label{exp}
Let $\Psi\in C^2(\RR^d)$ satisfy Assumption~\ref{as:subquadratic}. Let $\rho_{0} \in \mathcal{P}_{2}(\RR^d)$ with $\mathcal{F}(\rho_{0})<\infty$, and assume $\rho_{(\cdot)}\in C(\rho_0,\rho)$ with $\tilde{J}_\tau(\rho_{(\cdot)})$ finite. Then $\rho_{t} \in \Pc_{2}(\RR^{d})$ for
every $t$. Furthermore, the curve $\rho_{(\cdot)}$ lies in $\AC_{W_{2}} \Kl [0,1];\Ptwo(\RR^{d})\Kr$ and $\mathcal{F}(\rho_{t})$ is absolutely continuous with respect to $t$. Finally there holds:
\begin{multline*}
\frac{1}{4\tau}\int_{0}^{1}\left\|\frac{\partial\rho_{t}}{\partial t}-\tau(\Delta\rho_{t}+\div(\rho_{t}\nabla\Psi))\right\|^{2}_{-1,\rho_{t}}\,dt \\=
\frac{1}{4\tau}\int_{0}^{1}\left\|\frac{\partial\rho_{t}}{\partial t}\right\|^{2}_{-1,\rho_{t}}dt + \frac{\tau}{4}\int_{0}^{1}\left\|\grad \mathcal{\Ff}(\rho_t)\right\|^{2}_{-1,\rho_{t}}\,dt +\frac{1}{2}\mathcal{F}(\rho_{1})-\frac{1}{2}\mathcal{F}(\rho_{0}).
\end{multline*}
\end{proposition}

Before we prove this theorem we prove two auxiliary lemmas.
\begin{lemma}\label{lem:aux1} Assume that
\begin{enumerate}
\item $\Psi\in C^2(\RR^d)$ satisfies Assumption~\ref{as:subquadratic},
\item $\int\Psi\rho_0(dx)<\infty$,
\item $\rho_{(\cdot)} \in C(\rho_0,\rho)$,
\item $\tilde{J}_\tau(\rho_{(\cdot)})<\infty$.
\end{enumerate}
Then
\begin{equation}
\int_{0}^{\tau}\int_{\RR^d}|\nabla\Psi(x)|^{2}\rho_{t}(dx)\,dt<\infty. \label{eq:finiteNablaPsi}
\end{equation}
\end{lemma}
\begin{proof}
For simplicity we take $\tau=1$. We will prove the following statement: there exist $0<\delta\leq1$ and $\alpha,\beta>0$ that depend only on $\Psi$  such that
\begin{equation}
 \label{eq:finiteNablaPsi2}
  \alpha \sup_{t\in[0,\delta]} \int_{\RR^d}\!|\Psi|\,d\rho_t + \beta \int_0^\delta\int_{\RR^d}\! |\nabla\Psi|^2\,d\rho_t\,dt \leq 8\tilde{J}_1(\rho_{(\cdot)}) + \frac4e|\inf\Psi| + \frac2e \int_{\RR^d}\!\Psi\,d\rho_0 + \frac {2\delta} e\|\Delta\Psi\|_\infty. \\
\end{equation}
Obviously \eqref{eq:finiteNablaPsi} follows from \eqref{eq:finiteNablaPsi2} by repeating it $1/\delta$ times.

We will approximate $\Psi$ by a sequence of $C^2_{c}(\RR^d)$ functions which are allowed in the definition of the norm $\|\cdot\|_{-1}$. To account for the compact support we use the usual bump function:
\begin{equation*}
  \eta(x):=\begin{cases}
    \exp\left(\frac{-1}{1-|x|^2}\right),   &|x|\leq 1\\
    0,                                    &|x|>1.
  \end{cases}
\end{equation*}
Define $\eta_k(x):=\eta(x/k)$. Then the following estimates hold
\begin{align}
\label{eq:etak estimates}
  |\eta_k(x)|\leq 1/e,&&  |\nabla\eta_k(x)|\leq \frac1{k}&& \text{and} &&|\Delta\eta_k(x)|\leq\frac1{k^2}\|\Delta\eta\|_\infty<\infty.
\end{align}
Since $\eta_k\Psi\in \mathcal{D}$ the rate functional \eqref{eq:ldp sub and super} is bounded from below by
\begin{equation}
\label{eq:varphiK}
\begin{split}
  4\tilde{J}_1(\rho_{(\cdot)})&=\int_0^1 \sup_{p\in\mathcal{D}} \left(\langle\partial_t\rho_t-\Delta\rho_t-\div(\rho_t\nabla \Psi),p\rangle-\frac{1}{2}\int_{\RR^d}|\nabla p|^2d\rho_t\right)dt
\\&\geq \int_0^s\left(\langle\partial_t\rho_t-\Delta\rho_t-\div(\rho_t\nabla\Psi),\eta_k\Psi \rangle-\frac{1}{2}\int_{\RR^d}| \nabla(\eta_k\Psi)|^2d\rho_t\right)dt
\end{split}
\end{equation}
for any $s\in[0,1]$. We now estimate each term in the right-hand side of~\eqref{eq:varphiK}. For the first term, we have
\begin{equation}
  \int_0^s\!\langle\partial_t\rho_t,\eta_k\Psi\rangle\, dt=\int_{\RR^d}\!\eta_k\Psi\,d\rho_s-\int_{\RR^d}\!\eta_k\Psi\,d\rho_0 \geq \int_{\RR^d}\eta_k|\Psi|\,d\rho_s-\frac2e|\inf\Psi| -\int_{\RR^d}\eta_k\Psi\,d\rho_0.
\label{eq:part1}
\end{equation}
For the second part, we find
\begin{equation}
\begin{split}
-\int_0^s\langle \Delta\rho_t,\eta_k\Psi\rangle\, dt&=-\int_0^s\int_{\RR^d}\!\left( \Psi \Delta\eta_k + 2\nabla\eta_k\cdot\nabla\Psi+\eta_k\Delta\Psi\right) d\rho_t\,dt\\
  &\geq-\int_0^s\int_{\RR^d}\!\left(|\Delta\eta_k|\,|\Psi|+ |\nabla\eta_k|\,(|\nabla\Psi|^2+1)+\eta_k|\Delta\Psi| \right) d\rho_t\,dt\\
  &\mathop{\geq}^{\eqref{eq:etak estimates}}  -\int_0^s\int_{[-k,k]^d}\!\left( \frac1{k^2}\|\Delta\eta\|_\infty|\Psi|+\frac1k(|\nabla\Psi|^2+1)+\frac1e|\Delta\Psi| \right) d\rho_t\,dt\\
  &\geq-\frac s{k^2}\|\Delta\eta\|_\infty \sup_{t\in[0,s]} \int_{[-k,k]^d}\!|\Psi|\,d\rho_t-\frac1k\int_0^s\int_{[-k,k]^d}\!|\nabla\Psi|^2\,d\rho_t\,dt-\frac sk -\frac se\|\Delta\Psi\|_\infty.
\label{eq:part2}
\end{split}
\end{equation}
Finally, for the last part
\begin{equation}
\label{eq:part3}
\begin{split}
&\int_0^s\left(\langle-\div(\rho_t\nabla\Psi),\eta_k\Psi\rangle-\frac{1}{2}\int_{\RR^d}|\nabla(\eta_k\Psi)|^2d\rho_t\right)dt
\\&\qquad=\int_0^s\int_{\RR^d}\left( -\frac12|\nabla\eta_k|^2\Psi^2 + (1-\eta_k)\nabla\eta_k\cdot\Psi\nabla\Psi + (1-\frac12\eta_k)\eta_k|\nabla\Psi|^2\right)d\rho_t\,dt
\\&\qquad\mathop{\geq}^{\eqref{eq:etak estimates}} \int_0^s\int_{[-k,k]^d}\left( -\frac1{2k^2} \Psi^2 -|\frac2k\Psi|\,|\frac12\nabla\Psi| + \frac34\eta_k|\nabla\Psi|^2 \right)d\rho_t\,dt,
\\&\qquad\geq \int_0^s\int_{[-k,k]^d}\left( -\frac5{2k^2} \Psi^2 + \left(\frac34\eta_k-\frac18\right)|\nabla\Psi|^2 \right)d\rho_t\,dt.
\\&\qquad\geq \int_0^s\int_{[-k,k]^d}\!\left( -\frac{5C(1+k^2)}{2k^2} |\Psi| + \left(\frac34\eta_k-\frac18\right)|\nabla\Psi|^2 \right)d\rho_t\,dt
\\&\qquad\geq -\frac{5sC(1+k^2)}{2k^2}\sup_{t\in[0,s]} \int_{[-k,k]^d}\!|\Psi|\,d\rho_t  + \int_0^s\int_{[-k,k]^d}\left( \left(\frac34\eta_k-\frac18\right)|\nabla\Psi|^2 \right)d\rho_t\,dt,
\end{split}
\end{equation}
where the fourth line follows from Young's inequality, and in the fifth line we used the subquadratic assumption. Substituting \eqref{eq:part1}, \eqref{eq:part2} and \eqref{eq:part3} into \eqref{eq:varphiK} we get
\begin{multline*}
  \int_{\RR^d}\!\eta_k|\Psi|\,d\rho_s
   + \int_0^s\int_{[-k,k]^d}\! \frac34\eta_k |\nabla\Psi|^2\,d\rho_t\,dt
   \leq 4\tilde{J}_1(\rho_{(\cdot)}) + \frac2e|\inf\Psi| + \int_{\RR^d}\!\eta_k\Psi\,d\rho_0 + \frac sk + \frac se\|\Delta\Psi\|_\infty \\
  + \left(\frac s{k^2}\|\Delta\eta\|_\infty + \frac{5sC(1+k^2)}{2k^2}\right) \sup_{t\in[0,s]} \int_{[-k,k]^d}\!|\Psi|\,d\rho_t
  +\left(\frac18+\frac1k\right)\int_0^s\int_{[-k,k]^d}\!|\nabla\Psi|^2\,d\rho_t\,dt.
\end{multline*}
If we first discard the first term on the left-hand side and maximise the equation over $s\in[0,\delta]$ for some $0<\delta\leq1$, then discard the second term and maximise, the sum of the inequalities can be written as
\begin{multline}
\label{eq:aux after sup t}
  \sup_{t\in[0,\delta]} \int_{\RR^d}\!\left(\eta_k-\frac {2\delta}{k^2}\|\Delta\eta\|_\infty - \frac{5\delta C(1+k^2)}{k^2}\right)|\Psi|\,d\rho_t
   + \int_0^\delta\int_{[-k,k]^d}\! \left(\frac34\eta_k-\frac14-\frac{2}k\right) |\nabla\Psi|^2\,d\rho_t\,dt \\
   \leq 8\tilde{J}_1(\rho_{(\cdot)}) + \frac4e|\inf\Psi| + 2\int_{\RR^d}\!\eta_k\Psi\,d\rho_0 + \frac {2\delta} k + \frac {2\delta} e\|\Delta\Psi\|_\infty.
\end{multline}
Taking the supremum over $k\geq1$, the inequality \eqref{eq:aux after sup t} becomes
\begin{equation*}
\begin{split}
  &\big(\underbrace{\frac1e-5\delta C}_{:=\alpha}\big) \sup_{t\in[0,\delta]} \int_{\RR^d}\!|\Psi|\,d\rho_t
   + \big(\underbrace{\frac3{4e}-\frac14}_{:=\beta}\big)\int_0^\delta\int_{\RR^d}\! |\nabla\Psi|^2\,d\rho_t\,dt \\
  &\hspace{6cm} \leq 8\tilde{J}_1(\rho_{(\cdot)}) + \frac4e|\inf\Psi| + 2\sup_k\left\{\int_{\RR^d}\!\eta_k\Psi\,d\rho_0\right\} + 2\delta + \frac {2\delta} e\|\Delta\Psi\|_\infty \\
  &\hspace{6cm} \leq 8\tilde{J}_1(\rho_{(\cdot)}) + \frac8e|\inf\Psi| + \frac2e \int_{\RR^d}\!\Psi\,d\rho_0 + 2\delta + \frac {2\delta} e\|\Delta\Psi\|_\infty, \\
\end{split}
\end{equation*}
as $\sup_k\int\!\eta_k\Psi\,d\rho_0 \leq \sup_k\int\!\eta_k|\Psi|\,d\rho_0 \leq 1/e\int\!|\Psi|\,d\rho_0 \leq 1/e \int\!(\Psi+2|\inf\Psi|)\,d\rho_0$. Take $\delta$ such that $\alpha>0$.
Now that we know that the suprema are finite, we can take the limit $k\to\infty$ of \eqref{eq:aux after sup t}, which proves \eqref{eq:finiteNablaPsi2}.
\end{proof}

The second auxiliary lemma is:
\begin{lemma}\label{Ibound}
Let $\epsilon>0$ and $\rho(x)\,dx\in\Pc(\RR^d)$ be given. Let $\theta(x):=\Kl\frac{1}{2\pi}\Kr^{\frac{d}{2}}e^{\frac{-|x|^2}{2}}$ be the density of
the d-dimensional normal distribution. We define $\theta_{\epsilon}(x):=\epsilon^{-d}\theta(\frac{x}{\epsilon})$ and
$\rho_{\epsilon}:=\rho*\theta_{\epsilon}$. Then there exists a constant $C_{\epsilon}$ that depends only on $\epsilon$ such that
$\|\Delta(\rho_{\epsilon})\|^2_{-1,\rho_{\epsilon}}<C_{\epsilon}$.
\end{lemma}
\begin{proof}
 We have $$\nabla\rho_{\epsilon}(x)=(\rho*\nabla \theta_{\epsilon})(x)=\int_{\RR^{d}}\rho(x-y)\nabla \theta_{\epsilon}(y)\,dy=-\epsilon^{-2}\int_{\RR^{d}}\rho(x-y)y\theta_{\epsilon}(y)\,dy.$$
Furthermore $$|\nabla\rho_{\epsilon}(x)|^{2}\leq \epsilon^{-4}\int_{\RR^{d}}\rho(x-y)|y|^{2}\theta_{\epsilon}(y)\,dy
\int_{\RR^{d}}\rho(x-y)\theta_{\epsilon}(y)\,dy\leq \epsilon^{-4}\rho_{\epsilon}(x)\int_{\RR^{d}}\rho(x-y)|y|^{2}\theta_{\epsilon}(y)\,dy.$$
Now
\begin{align*}
\|\Delta(\rho_{\epsilon})\|^2_{-1,\rho_{\epsilon}}=\int_{\RR^{d}}\!\frac{|\nabla \rho_{\epsilon}(x)|^{2}}{\rho_{\epsilon}(x)}\,dx&\leq \epsilon^{-4}\int_{\RR^{d}}\int_{\RR^{d}}\rho(x-y)|y|^{2}\theta_{\epsilon}(y)\,dy\,dx
\\&=\epsilon^{-4}\int_{\RR^{d}}\int_{\RR^{d}}\rho(x-y)\,dx\,|y|^{2}\theta_{\epsilon}(y)\,dy
\\&\leq \epsilon^{-4} \int_{\RR^{d}}|y|^{2}\theta_{\epsilon}(y)\,dy:=C_\epsilon.
\end{align*}
\end{proof}

We are now ready to proceed with the
\begin{proof}[Proof of Proposition ~\ref{exp}]
 Let $\rho_{(\cdot)}$ satisfy the assumptions (of Proposition~\ref{exp}). By Lemma~\ref{lem:aux1} we have $$\int_{0}^{1}\int_{\RR^{d}}|\nabla\Psi(x)|^{2}\rho_{t}(dx)\,dt<\infty$$ and therefore
\begin{multline*}
  \frac{1}{4\tau}\int_{0}^{1}\|\frac{\partial\rho_{t}}{\partial t}-\tau\Delta\rho_{t}\|^{2}_{-1,\rho_{t}}\,dt <\frac{1}{2\tau}\int_{0}^{1}\left\|\frac{\partial\rho_{t}}{\partial t}-
\tau(\Delta\rho_{t}+\div(\rho_{t}\nabla \Psi))\right\|^{2}_{-1,\rho_{t}}\,dt\\
+\frac{\tau}{2}\int_{0}^{1}\int_{\RR^{d}}|\nabla\Psi|^{2}\rho_{t}(dx)\,dt<\infty.
\end{multline*}
Take a $0<s\leq 1$. Since
\begin{equation}\label{grad}
 \frac{1}{4\tau}\int_{0}^{s}\left\|\frac{\partial\rho_{t}}{\partial t}-\tau\Delta\rho_{t}\right\|^{2}_{-1,\rho_{t}}\,dt<\infty
\end{equation}
we have that $\|\frac{\partial\rho_{t}}{\partial t}-\tau\Delta\rho_{t}\|^{2}_{-1,\rho_{t}}<\infty$ for almost every $t$. By \cite[Lem. D.34]{Feng2006} there is a $v_{t}\in L^{2}(\rho_{t})$ such that $$\frac{\partial\rho_{t}}{\partial t}-\tau\Delta\rho_{t}=-\div(v_{t}\,\rho_{t})$$ in distributional
sense. Take $\theta_{\epsilon}(x)$ as in Lemma \ref{Ibound}. Then we have  $$\frac{\partial\rho_{t,\epsilon}}{\partial t}-\tau\Delta\rho_{t,\epsilon}=-\div(v_{t,\epsilon}\,\rho_{t,\epsilon}),$$ where
 $$\rho_{t,\epsilon}=\rho_{t}*\theta_{\epsilon}(x),\hhhh v_{t,\epsilon}=\frac{(v_{t}\,\rho_{t})*\theta_{\epsilon}(x)}{\rho_{t,\epsilon}}.$$

By \cite[Th. 8.1.9]{Ambrosio2008} we have
\begin{multline}
\label{eq:AGS approximation}
\frac{1}{4\tau}\int_{0}^{s}\!\big\|\frac{\partial\rho_{t,\epsilon}}{\partial t}-\tau\Delta\rho_{t,\epsilon}\big\|^{2}_{-1,\rho_{t}}\,dt=
\frac{1}{4\tau}\int_{0}^{s}\!\|v_{t,\epsilon}\|^{2}_{L^{2}(\rho_{t,\epsilon})}\,dt
\leq\frac{1}{4\tau}\int_{0}^{s}\!\|v_{t}\|^{2}_{L^{2}(\rho_{t})}\,dt=\\\frac{1}{4\tau}\int_{0}^{s}\left\|\frac{\partial\rho_{t}}{\partial t}-\tau\Delta\rho_{t}\right\|^{2}_{-1,\rho_{t}}\,dt.
\end{multline}

Furthermore by Lemma~\ref{Ibound} we have that
\begin{equation}\label{eq:finiteDelta}
\int_{0}^{s}\!\|\Delta\rho_{t,\epsilon}\|^{2}_{-1,\rho_{t,\epsilon}}\,dt\leq C_{\epsilon},
\end{equation}
and therefore
\begin{equation}\label{eq:finiteFisher2}
\int_{0}^{s}\big\|\frac{\partial\rho_{t,\epsilon}}{\partial t}\big\|^{2}_{-1,\rho_{t,\epsilon}}\,dt<\infty.
\end{equation}
From \eqref{eq:finiteDelta} and since $\rho(0)\in \Pc_2(\RR^d)$, by using \cite[Lem. D.34]{Feng2006} and Lemma~\ref{ACcurve} we get that the curve $\rho_{t,\epsilon}$ is absolutely continuous in $\Pc_{2}(\RR^{d})$. In addition, it is a straightforward that $\mathcal{S}(\rho_{t,\epsilon})$ is finite for every $0<t\leq s$. From \eqref{eq:finiteDelta}, \eqref{eq:finiteFisher2} and by Lemma \ref{chainrule},
$\mathcal{S}(\rho_{t,\epsilon})$ is absolutely continuous with respect to $t$. Hence we obtain
\begin{align*}
&\frac{1}{4\tau}\int_{0}^{s}\big\|\frac{\partial\rho_{t,\epsilon}}{\partial t}+\tau \Delta\rho_{t,\epsilon}\big\|^{2}_{-1,\rho_{t}} dt
\\&=\frac{1}{4\tau}\int_{0}^{s}\big\|\frac{\partial\rho_{t,\epsilon}}{\partial t}\big\|^{2}_{-1,\rho_{t}} dt
+\frac{\tau}{4}\int_{0}^{s}\|\Delta\rho_{t,\epsilon}\|^{2}_{-1,\rho_{t}} dt
+\frac{1}{2}\int_{0}^{s} \left( \grad \mathcal{S}(\rho_{t,\epsilon}), \frac{\partial\rho_{t,\epsilon}}{\partial t}\right)_{-1,\rho_{t}}\,dt
\\&= \frac{1}{4\tau}\int_{0}^{s}\big\|\frac{\partial\rho_{t,\epsilon}}{\partial t}\big\|^{2}_{-1,\rho_{t}}dt
+\frac{\tau}{4}\int_{0}^{s}\|\Delta\rho_{t,\epsilon}\|^{2}_{-1,\rho_{t}}+\frac{1}{2}\mathcal{S}(\rho_{s,\epsilon})-\frac{1}{2}\mathcal{S}(\rho_{0,\epsilon}). \end{align*}
It follows that
\begin{equation*}
\frac{1}{4\tau}\int_{0}^{s}\big\|\frac{\partial\rho_{t,\epsilon}}{\partial t}\big\|^{2}_{-1,\rho_{t}}dt
+\frac{\tau}{4}\int_{0}^{s}\|\Delta\rho_{t,\epsilon}\|^{2}_{-1,\rho_{t}}dt+\frac{1}{2}\mathcal{S}(\rho_{s,\epsilon})-\frac{1}{2}\mathcal{S}(\rho_{0,\epsilon})
\leq \frac{1}{4\tau}\int_{0}^{1}\big\|\frac{\partial\rho_{t}}{\partial t}-\tau\Delta\rho_{t}\big\|^{2}_{-1,\rho_{t}}dt.
\end{equation*}
Now letting $\epsilon$ go to zero and by the lower semicontinuity of the entropy and the Fisher information functionals we get $\mathcal{S}(\rho_{s})<\infty $ and $ \int_{0}^{s}\|\Delta\rho_{t}\|^{2}_{-1,\rho_{t}}\,dt<\infty $. Therefore
\begin{equation*}
\int_0^s\big\|\frac{\partial\rho_t}{\partial t}\big\|^{2}_{-1,\rho_{t}}dt\leq 2\left(\int_0^s\big\|\frac{\partial\rho_{t}}{\partial t}-\tau\Delta\rho_{t}\big\|^{2}_{-1,\rho_{t}}dt+\tau^2 \int_{0}^{s}\|\Delta\rho_{t}\|^{2}_{-1,\rho_{t}}\,dt\right)<\infty.
\end{equation*}
and

\begin{equation*}
\int_0^s\big\|\Delta\rho_t+\div{\rho_t\nabla\Psi}\big\|^{2}_{-1,\rho_{t}}dt\leq 2\left(\int_{0}^{s}\|\Delta\rho_{t}\|^{2}_{-1,\rho_{t}}\,dt+\int_{0}^{s}\int_{\RR^d}|\nabla\Psi(x)|^{2}\rho_{t}(x)\,dx\,dt\right)<\infty.
\end{equation*}
By Lemma \ref{ACcurve} and Lemma \ref{chainrule} again, the curve $\rho_t$ is in $ AC_{W_{2}} \Kl [0,1];\Ptwo(\RR^{d})\Kr $. Moreover, $ t\mapsto\mathcal{F}(\rho_{t})$ is absolutely continuous and \eqref{eq:chainrule} holds.  Hence we have
\begin{multline*}
\frac{1}{4\tau}\int_{0}^{1}\left\|\frac{\partial\rho_{t}}{\partial t}-\tau(\Delta\rho_{t}+\div(\rho_{t}\nabla\Psi))\right\|^{2}_{-1,\rho_{t}}\,dt \\=
\frac{1}{4\tau}\int_{0}^{1}\left\|\frac{\partial\rho_{t}}{\partial t}\right\|^{2}_{-1,\rho_{t}}dt + \frac{\tau}{4}\int_{0}^{1}\left\|\Delta\rho_{t}+\div(\rho_{t}\nabla\Psi))\right\|^{2}_{-1,\rho_{t}}\,dt +\frac{1}{2}\mathcal{F}(\rho_{1})-\frac{1}{2}\mathcal{F}(\rho_{0}).
\end{multline*}
This finishes the proof of the Lemma.
\end{proof}
\begin{remark}
For the superquadratic case, the above lemma was proved by Feng and Nguyen in \cite{Feng2011} by using probabilistic tools. In addition, they obtain an estimate for the growth of $\mathcal{F}$ along the curves.
\end{remark}

Now the following is a straightforward result:
\begin{corollary} Let $\rho_{0}\in \Pc_{2}(\RR^{d})$ with $\mathcal{F}(\rho_{0})<\infty$. If $\Psi \in C^{2}(\RR^{d})$ satisfies either Assumption~\ref{as:subquadratic} or \ref{as:superquadratic}, then
\begin{equation*}
J_{\tau}(\rho|\rho_0)=\inf_{\rho_{(\cdot)}\in C_{W_2}(\rho_{0},\rho)}\frac{1}{4\tau} \int_0^1\big\|\frac{\partial \rho_{t}}{\partial t}-\tau(\Delta\rho_{t}+\div(\rho_{t}\nabla\Psi))\big\|^2_{-1,\rho_{t}}\,dt.
\end{equation*}
\end{corollary}

\section{Lower bound}
\label{sec:lower bound}
In this section we prove the lower bound of the Gamma convergence~\eqref{eq:main result} in our main result, Theorem~\ref{th:main result}.
\begin{theorem}[Lower bound]
\label{th:lower bound} Under the assumptions of Theorem~\ref{th:main result}, we have for any $\rho_1 \in \Pc_2(\RR^d)$ and all sequences $\rho_1^\tau\in\Ptwo(\RR^d)$ narrowly converging to $\rho_1$
\begin{equation}
\label{eq:lower bound2}
    \liminf_{\tau\rightarrow 0} \left(J_{\tau}(\rho_1^\tau|\rho_0)-\frac{W_{2}^{2}(\rho_{0},\rho_1^\tau)}{4\tau}\right) \geq \frac{1}{2}\mathcal{F}(\rho_1)-\frac{1}{2}\mathcal{F}(\rho_{0}).
\end{equation}
\end{theorem}

\begin{proof}
Take any sequence $\rho_1^\tau\in\Pc_2(\RR^{d})$ narrowly converging to a $\rho_1 \in \Pc_2(\RR^{d})$. We only need to consider those $\rho_1^\tau$ for which $J_{\tau}(\rho_1^\tau|\rho_0)< \infty$. For each such $\rho_1^\tau$, by the definition of infimum there exists a curve $\rho_t^\tau\in C(\rho_0,\rho_1^\tau)$ satisfying
\begin{equation}
\frac{1}{4\tau}\int_{0}^{1}\big\| \frac{\partial \rho_{t}^\tau}{\partial t}-\tau(\Delta\rho_{t}^\tau+\div(\rho_{t}^\tau\,\nabla\Psi)) \big\|^{2}_{-1,\rho_{t}^\tau}\,dt\leq J_\tau(\rho_1^\tau|\rho_0)+\tau < \infty.
\end{equation}
By Lemma~\ref{exp} for the subquadratic case and \cite[Lem. 2.6]{Feng2011} for the superquadratic
case, we have
\begin{align*}
J_\tau(\rho_1^\tau|\rho_0)+\tau &\geq\frac{1}{4\tau}\int_{0}^{1}\big\| \frac{\partial \rho_{t}^\tau}{\partial t}-\tau(\Delta\rho_{t}^\tau+\div(\rho_{t}^\tau\,\nabla\Psi))\big\|^{2}_{-1,\rho_{t}^\tau}\,dt
\\&=\frac{1}{4\tau}\int_{0}^{1}\big\| \frac{\partial \rho_{t}^\tau}{\partial t}+\tau\grad \mathcal{F}(\rho_t^\tau)) \big\|^{2}_{-1,\rho_{t}^\tau}\,dt
\\&=\frac{1}{4\tau}\left(\int_{0}^{1}\big\| \frac{\partial \rho_{t}^\tau}{\partial t}\big\|^2_{-1,\rho_{t}^\tau}\,dt+2\tau(\mathcal{F}(\rho_1^\tau)-\mathcal{F}(\rho_0))+\tau^2\int_{0}^{1}\|\grad \mathcal{F}(\rho_t^\tau)) \|^{2}_{-1,\rho_{t}^\tau}\,dt\right)
\\&=\frac{1}{2}(\mathcal{F}(\rho_1^\tau)-\mathcal{F}(\rho_0))+\frac{1}{4\tau}\int_{0}^{1}\big\|\frac{\partial \rho_{t}^\tau}{\partial t}\big\|^2_{-1,\rho_{t}^\tau}\,dt+\frac{\tau}{4}\int_{0}^{1}\|\grad \mathcal{F}(\rho_t^\tau)) \|^{2}_{-1,\rho_{t}^\tau}\,dt
\\&\geq\frac{1}{2}(\mathcal{F}(\rho_1^\tau)-\mathcal{F}(\rho_0))+\frac{1}{4\tau}\int_{0}^{1}\big\| \frac{\partial \rho_{t}^\tau}{\partial t}\big\|^2_{-1,\rho_{t}^\tau}\,dt
\\&\geq \frac{1}{2}(\mathcal{F}(\rho_1^\tau)-\mathcal{F}(\rho_0))+\frac{1}{4\tau}W_{2}^{2}(\rho_{0},\rho_1^\tau).
\end{align*}
In the last inequality above we have used the Benamou-Brenier formula for the Wasserstein distance \cite{Benamou2000}.
Finally, using $\rho_1^\tau\to\rho_1$ narrowly with the narrow lower semi-continuity of $\mathcal{F}$, we find that
\begin{equation*}
  \liminf_{\tau\rightarrow 0} \left(J_{\tau}(\rho_1^\tau|\rho_0)-\frac{W_{2}^{2}(\rho_{0},\rho_1^\tau)}{4\tau} \right)\geq \frac{1}{2}\mathcal{F}(\rho_1)-\frac{1}{2}\mathcal{F}(\rho_{0}).
\end{equation*}
\end{proof}

\section{Recovery sequence}
\label{sec:recovery sequence}

In this section we prove the upper bound of the Gamma convergence~\eqref{eq:main result}. This will conclude the proof of Theorem~\ref{th:main result}.
\begin{theorem}[Recovery sequence] \label{th:recovery sequence}
Under the assumptions of Theorem~\ref{th:main result}, for any $\rho_1\in\Ptwo(\RR)$ there exists a sequence $\rho_1^\tau\in\Ptwo(\RR)$ converging to $\rho_1$ in the Wasserstein metric such that
\begin{equation}
\label{eq:recovery sequence2}
    \limsup_{h\rightarrow 0}\left( J_{h}(\rho_1^\tau|\rho_0)-\frac{W_{2}^{2}(\rho_{0},\rho_1^\tau)}{4h}\right) \leq \frac{1}{2}S(\rho_1)-\frac{1}{2}S(\rho_{0}).
\end{equation}
\end{theorem}

As mentioned in Section~\ref{sec:introduction}, our approach for the recovery sequence only works for $d=1$. Hence throughout this section, we will consider $d=1$.

The existence of the recovery sequence is proven by making use of the following denseness argument, which is also interesting in its own\footnote{A more or less similar idea can be found in \cite[Remark 1.29]{Braides2002}; Proposition~\ref{th:dense recovery} is slightly stronger.}:
\begin{proposition}\label{th:dense recovery}
Let $(X,d)$ be a metric space and let $Q$ be a dense subset of $X$. If $\{K_{n}, n \in \NN\}$ and $K_\infty$ are functions from $X$ to $\RR$ such that:
\begin{enumerate}[(a)]
  \item $K_{n}(q)\rightarrow K_{\infty}(q)$ for all $q\in Q$,
  \item for every $x\in X$ there exists a sequence $q_{n}\in Q$ with $q_{n}\rightarrow x$ and $K_{\infty}(q_{n})\rightarrow K_{\infty}(x)$,
\end{enumerate}
then for every $x\in X$ there exists a sequence $r_{n}\in Q$, with $r_{n}\rightarrow x$ such that $K_{n}(r_{n})\rightarrow K_{\infty}(x)$.
\end{proposition}
\begin{proof}
  The proof is by a diagonal argument. Take any $x\in X$ and take the corresponding sequence $q_n\to x$ such that $K_{\infty}(q_{n})\rightarrow K_{\infty}(x)$.
  By assumption, for any $q\in Q$ and $L>0$ there exists a $n_{L,q}$ such that for any $n\geq n_{L,q}$ there holds $d(K_n(q),K_\infty(q))<1/L$. Define
  \begin{equation*}
  l_n:= \begin{cases}
          1,  &1\leq n < n_{2,q_2},\\
          2,  &n_{2,q_2}\leq n < \max\{n_{2,q_2},n_{3,q_3}\},\\
          \hdots
        \end{cases}
  \end{equation*}
  Take the subsequence $r_n:=q_{l_n}$. Observe that $l_n\to \infty$ as $n\to\infty$ such that indeed $q_{l_n}\to x$, and:
  \begin{equation*}
    d(K_n(q_{l_n}),K_\infty(x)) \,\leq\, \underbrace{d(K_n(q_{l_n}),K_\infty(q_{l_n}))}_{\leq \frac1{l_n}}+d(K_\infty(q_{l_n}),K_\infty(x)) \to 0.
  \end{equation*}
\end{proof}

For a fixed $\rho_0$ satisfying the assumptions of Theorem~\ref{th:main result}, we want to apply Proposition~\ref{th:dense recovery} to the situation where
\begin{align*}
  &X=\Ptwo(\RR), \\
  &Q=Q(\rho_0)= \Big\{\rho=\rho(x)dx \in \mathcal{P}_{2}(\RR): \rho(x) \text{ is bounded from below by a positive constant in every compact set}, \\&\qquad \qquad\qquad \qquad\mathcal{F}(\rho),\|\Delta\rho\|^{2}_{-1,\rho},\int_{\RR}|\nabla\Psi(x)|^{2}\rho(x)\,dx<\infty, \text{ and there exists a } M>0 \text{ such that }
\\&\qquad \qquad\qquad\qquad\qquad \qquad \rho_0(x)=\rho(x) \text{ for all } |x|>M\Big\},\\
  &K_n(\rho)=J_{h_{n}}(\rho\,|\rho_0)-\frac{W_{2}^{2}(\rho_{0},\rho)}{4h_{n}}, \hspace{4pt} \text{where} \hspace{4pt} h_{n} \hspace{4pt} \text{an arbitrary sequence converging to zero},\\
  &K_\infty(\rho)=\frac12 \mathcal{F}(\rho)-\frac12 \mathcal{F}(\rho_0).
\end{align*}

Assumption~\textit{(a)} of Proposition~\ref{th:dense recovery}, i.e. pointwise convergence for every $\rho_1 \in Q(\rho_{0})$, can be proven as follows. Take $\rho_{1} \in Q(\rho_{0})$ and let $\rho_{t}$ be the geodesic that connects  $\rho_{0}$ and $\rho_{1}$. In the following Lemma~\ref{th:uniformly bounded}, we will prove that $\|\Delta\rho_{t}\|^{2}_{-1,\rho_{t}}$ and $\int_{\RR}|\nabla\Psi(x)|^{2}\rho_{t}(x)dx$ are uniformly bounded, so that we have
\begin{multline*}
  \int_{0}^{1}\left\|\frac{\partial \rho_{t}}{\partial t}-\tau(\Delta\rho_{t}+\div(\rho_{t}\nabla\Psi))\right\|^2_{-1,\rho_{t}}\,dt\\
    \leq 3\int_{0}^{1}\left\|\frac{\partial \rho_{t}}{\partial t}\right\|^2_{-1,\rho_{t}}\,dt +3\tau^2\int_{0}^{1}\|\Delta\rho_{t}\|^2_{-1,\rho_{t}}\,dt+3\tau^2\int_{0}^{1}\|\div(\rho_{t}\nabla\Psi)\|^2_{-1,\rho_{t}}\,dt< \infty.
\end{multline*}
By applying Lemma~\ref{exp} for the subquadratic case or \cite[Lem. 2.6]{Feng2011} for the superquadratic
case:
\begin{align*}
\lim_{\tau\rightarrow 0}\Kl J_{\tau}(\rho_{1}|\rho_0)-\frac{W_{2}^{2}(\rho_{0},\rho_{1})}{4\tau}\Kr &\leq \lim_{\frac{\tau}{2}\rightarrow 0}\Bigg{[}  \tau\int_{0}^{1}\Kl\int_{\RR}\Kl\frac{(\rho'_{t}(x))^{2}}{\rho_{t}(x)}+|\nabla\Psi(x)|^{2}\rho_{t}(x)\Kr dx\Kr dt\\
&+\frac{1}{2}\mathcal{F}(\rho_{1})-\frac{1}{2}\mathcal{F}(\rho_{0})\Bigg{]} = \frac{1}{2}\mathcal{F}(\rho_{1})-\frac{1}{2}\mathcal{F}(\rho_{0}).
\end{align*}\\
The pointwise convergence then follows from this together with the lower bound~\eqref{eq:lower bound2}.

To prove the uniform bounds:
\begin{lemma}
\label{th:uniformly bounded}
Let $\Psi\in C^{2}(\RR)$  be convex. Let $\rho_{0}=\rho(x)dx\in\Pc_2(\RR)$  be asolutely continuous with respect to the Lesbegue measure, where $\rho(x)$ is bounded from below by a positive constant in every compact set. Let $\rho_{1}\in Q(\rho_{0})$ and $\rho_{t}$ be the geodesic that connects  $\rho_{0}$ and $\rho_{1}$. Assume that $\mathcal{F}(\rho_{0})$ , $\|\Delta\rho_{0}\|^{2}_{-1,\rho_{0}}$ and $\int_{\RR}|\nabla\Psi(x)|^{2}\rho_0(x)dx$ are all finite. Then
$\mathcal{F}(\rho_{t})$, $\|\Delta\rho_{t}\|^{2}_{-1,\rho_{t}}$ and $\int_{\RR}|\nabla\Psi(x)|^{2}\rho_{t}(x)\,dx $ are uniformly bounded with respect to $t$.
\end{lemma}
\begin{proof}
Let $T(x)$ be the optimal map that transports $\rho_0(dx)$ to $\rho_1(dx)$. The geodesic that connects $\rho_{0}$ and $\rho_{1}$ is defined by
\begin{equation*}
  \rho_t(x)=((1-t)x+tT(x))_\sharp \rho_0(x).
\end{equation*}
First we prove that $\|\Delta\rho_{t}\|^{2}_{-1,\rho_{t}}$ is uniformly bounded with respect to $t$. In the real line, the map $T(x)$ can be determined via the cumulative distribution functions as follows \cite[Section 2.2]{Villani2003}). Let $F(x)$ and $G(x)$ be respectively the cumulative distribution functions of $\rho(dx)$ and $\rho_1(dx)$, i.e.
\begin{equation*}
  F(x)=\int_{-\infty}^x\rho_0(x)\,dx; ~~~ G(x)=\int_{-\infty}^x\rho_1(x)\,dx.
\end{equation*}
Then $T=G^{-1}\circ F$. We have
\begin{equation}
  F(M)+\int_M^{+\infty}\rho_0(x)\,dx =G(M)+\int_M^{+\infty}\rho_1(x)\,dx = 1.\label{totalsum}
\end{equation}
From \eqref{totalsum} and by the assumption that $\rho_0(x)=\rho_1(x)$ for all $|x|>M$ we find that $F(M)=G(M)$. Hence for all $x$ such that $|x|>M$ we have
\begin{equation*}
  F(x)=F(M)+\int_M^x\rho_0(x)\,dx =G(M)+\int_M^x\rho_1(x)\,dx = G(x).
\end{equation*}
Consequentially, for all $x$ with $|x|>M$ we have $T(x)=(G^{-1}\circ F)(x) = x$. Therefore $T'(x)=1$ for all $|x|>M$. This, together with the fact that $T$ is a $C^1$ function, implies that $T'(x)$ is bounded. Moreover $T(x)$ satisfies the Monge -  Amp\`{e}re equation.
\begin{equation*}
  \rho_{0}(x)=\rho_{1}(T(x))T'(x).
\end{equation*}
or equivalently (since $\rho_1(x)>0)$,
\begin{equation}
  T'(x)=\frac{\rho_0(x)}{\rho_1(T(x))}.\label{derT}
\end{equation}
Since the densities $\rho_0, \rho_1$ are absolutely continuous (recall that $\sqrt\rho_{0},\sqrt\rho_{1}\in H^{1}(\RR)$) and $T'(x)$ in $C^{1}$ and strictly positive, we get
\begin{align*}
\frac{T''(x)}{T'(x)}&=(\log(T'(x)))'
\\&= \Kl\log(\rho_{0}(x))-\log(\rho_{1}(T(x))\Kr '
\\&=\frac{\rho'_{0}(x)}{\rho_{0}(x)} -\frac{\rho'_{1}(T(x))T'(x)}{\rho_{1}(T(x))}.
\end{align*}
Set $T_{t}(x)=tx +(1-t)T(x)$. For $ 0\leq t \leq 1$ we have
\begin{equation}
\label{eq:Ampere}
  \rho_{t}(x)=\rho_{1}(T_{t}(x))T_{t}'(x),
\end{equation}
Since $\rho_1(T_{t}(x))$ and $T_{t}'(x)$ are both absolutely continuous so is $\rho_{t}(x)$. Hence the derivative appeared in \eqref{FisherI} for $\|\Delta\rho_{t}\|^{2}_{-1,\rho_{t}}$ is the classical derivative. Substituting \eqref{eq:Ampere} into \eqref{FisherI} we get
\begin{align}
\int_{\RR}\frac{( \rho'_{t}( x))^{2}}{\rho_{t}( x) }\,dx & =\int_{\RR}\frac{\lbrack( \rho_{1}( T_{t}(x))T_{t}'( x))'\rbrack^{2}}{\rho_{1}(T_{t}(x))T_{t}'(x)} \,dx \nonumber\\
&=\int_{\RR}\frac{\lbrack\rho'_{1}( T_{t}(x))T_{t}'( x)^2 +\rho_1(T_t(x))T_{t}''(x)\rbrack^{2}}{\rho_{1}(T_{t}(x))T_{t}'(x)}\,dx\nonumber\\
& \leq 2 \int_{\RR}\frac{ (\rho'_{1}( T_{t}( x ) ))^{2}( T'_{t}( x))^{4}} {\rho_{1}( T_{t}( x)) T'_{t}( x ) }\,dx+ 2\int_{\RR}\frac{( \rho_{1} ( T_{t}( x )) T_{t}''( x ) )^{2} }{\rho_{1}( T_{t}( x)) T_{t}'( x) }\,dx \nonumber\\
& = 2 \int_{\RR}\frac{ (\rho'_{1}( T_{t}( x ) ))^{2}} {\rho_{1}( T_{t}( x))}(T'_{t}(x))^3\,dx + 2\int_{\RR}\rho_{1}(T_{t}(x))\frac{(T''_{t}(x))^{2}}{T_{t}'(x)}\,dx\label{geoestimate}
\end{align}
Note that in the inequality above we have used the Cauchy - Schwarz inequality $(a+b)^2\leq 2(a^2+b^2)$. To proceed we will estimate each term in the right hand side of \eqref{geoestimate} using the fact that $|T'(x)|$ is bounded and $\|\Delta\rho_{0}\|^{2}_{-1,\rho_{0}},\|\Delta\rho_{1}\|^{2}_{-1,\rho_{1}}<\infty$. For the first part we have
\begin{align}
\int_{\RR}\frac{ (\rho'_{1}( T_{t}( x ) ))^{2}} {\rho_{1}( T_{t}( x))}(T'_{t}(x))^3 \,dx &=\int_{\RR}\frac{ (\rho'_{1}( T_{t}( x ) ))^{2}} {\rho_{1}( T_{t}( x))}(T'_{t}(x))(T'_{t}(x))^{2} \,dx\nonumber
\\&\leq C^{2}\int_{\RR}\frac{ (\rho'_{1}( T_{t}( x ) ))^{2}} {\rho_{1}( T_{t}( x))}(T'_{t}(x))\,dx\nonumber
\\&=C^{2}\int_{\RR}\frac{(\rho'_{1}(x))^{2}}{\rho_{1}(x)}\,dx\nonumber
\\&=C^2\|\Delta\rho_{1}\|^{2}_{-1,\rho_{1}}.\label{estimatep1}
\end{align}
Let $B$ be the ball of radius $M$ centered at the origin. Since $T''(x)=0$ for all $|x|>M$ we can restrict our calculation for the second part in the ball $B$.
\begin{align}
\int_{\RR}\rho_{1}(T_{t}(x))\frac{(T''_{t}(x))^{2}}{T_{t}'(x)}\,dx&=\int_B\rho_{1}(T_{t}(x))\frac{(T''_{t}(x))^{2}}{T_{t}'(x)}\,dx\nonumber
\\&=\int_{B}\rho_{1}(T_{t}(x))\frac{((1-t)T''(x))^{2}}{T_{t}'(x)}\, dx\nonumber
\\&=\int_B\rho_{1}(T_{t}(x))T'_{t}(x)\Kl\frac{T'(x)(1-t)}{T'_{t}(x)}\Kr^{2}\Kl\frac{T''(x)}{T'(x)}\Kr^{2}\,dx\nonumber
\\&=\int_B\rho_{1}(T_{t}(x))T'_{t}(x)\Kl\frac{T'(x)(1-t)}{t+(1-t)T'(x)}\Kr^{2}\Kl\frac{\rho'_{0}(x)}{\rho_{0}(x)} -\frac{\rho'_{1}(T(x))T'(x)}{\rho_{1}(T(x))}\Kr^{2}\,dx\nonumber
\\&\leq  2\int_B\rho_{1}(T_{t}(x))T'_{t}(x)\Kl\frac{\rho'_{0}(x)}{\rho_{0}(x)}\Kr^2\,dx\nonumber
\\&+2\int_B\rho_{1}(T_{t}(x))T'_{t}(x)\Kl\frac{\rho'_{1}(T(x))T'(x)}{\rho_{1}(T(x))}\Kr^{2} dx\nonumber
\\&= 2\int_B\frac{\rho_{1}(T_{t}(x))T'_{t}(x)}{\rho_{0}(x)}\Kl\frac{(\rho'_{0}(x))^{2}}{\rho_{0}(x)}\Kr dx\nonumber
\\&~~~+2\int_B\frac{\rho_{1}(T_{t}(x))T'_{t}(x)T'(x)}{\rho_{1}(T(x))}\Kl\frac{(\rho'_{1}(T(x)))^{2}}{\rho_{1}(T(x))}T'(x)\Kr\,dx\nonumber
\\&\leq C\Kl\int_B\frac{(\rho'_{0}(x))^{2}}{\rho_{0}(x)}\,dx+\int_B\frac{(\rho'_{1}(T(x)))^{2}}{\rho_{1}(T(x))}T'(x)\,dx\Kr\nonumber
\\&\leq C (\|\Delta\rho_{0}\|^{2}_{-1,\rho_{0}}+\|\Delta\rho_{1}\|^{2}_{-1,\rho_{1}}). \label{estimatep2}
\end{align}
From (\ref{geoestimate}), (\ref{estimatep1}) and (\ref{estimatep2}) we find that
\begin{equation*}
\|\Delta\rho_{t}\|^{2}_{-1,\rho_{t}}=\int_{\RR}\frac{( \rho'_{t}( x))^{2}}{\rho_{t}( x) }\,dx\leq C (\|\Delta\rho_{0}\|^{2}_{-1,\rho_{0}}+\|\Delta\rho_{1}\|^{2}_{-1,\rho_{1}}).
\end{equation*}
It remains to prove the boundedness of the functional $\int_{\RR}|\nabla\Psi(x)|^{2}\rho_{t}(x)dx $. \\

Since $T(x)=x$ for $|x|>M$ we have $\rho_t(x)=\rho_1(x)$ for $|x|>M$. Hence
\begin{align*}
\int_{\RR}|\nabla\Psi(x)|^{2}\rho_{t}(x)\,dx&=\int_{B}|\nabla\Psi(x)|^{2}\rho_{t}(x)\,dx+\int_{|x|>M}|\nabla\Psi(x)|^{2}\rho_{t}(x)\,dx
\\&=\int_{B}|\nabla\Psi(x)|^{2}\rho_{t}(x)\,dx+\int_{|x|>M}|\nabla\Psi(x)|^{2}\rho_{1}(x)\,dx
\\ &\leq C\int_{B}\rho_{t}(x)\,dx+\int_{|x|>M}|\nabla\Psi(x)|^{2}\rho_{1}(x)\,dx
\\ &\leq C+\int|\nabla\Psi(x)|^{2}\rho_{1}(x)\,dx<\infty.
\end{align*}
Finally the result for $\mathcal{F}(\rho_{t})$ comes from the fact that $\mathcal{F}$ is geodesically convex.
\end{proof}

Finally, to prove assumption~\textit{(b)} of Proposition~\ref{th:dense recovery}, i.e. the existence of the recovery sequence in the dense set.
\begin{lemma}
Let $\rho_{0},\rho_{1} \in \mathcal{P}_{2}(\RR)$ and $\Psi \in C^{2}(\RR)$ with $\Psi(x) > -A -B|x|^{2}$ for some positive constants (this includes both our cases). Assume that $\rho_{0}$ is bounded from below by a positive constant in every compact set and that $\mathcal{F}(\rho_{0}), \|\Delta\rho_{0}\|^{2}_{-1,\rho_{0}}$ and $\int_{\RR^{n}}|\nabla\Psi(x)|^{2}\rho_{1}(x)\,dx $ are all finite. Then, there exists a sequence $k_{n}\in Q(\rho_{0})$  such that  $k_{n}\rightarrow \rho_{1} \hh$ with respect to Wasserstein distance, and $\mathcal{F}(k_{n})\rightarrow \mathcal{F}(\rho_{1})$.
\end{lemma}

\begin{proof}
We will assume that $\mathcal{F}(\rho_{1})<\infty$, otherwise the construction is trivial.
Let $n \in \NN$. Since $\int_{\RR}\rho_{0}(x)x^{2}\,dx<\infty$ and  $\int_{\RR}\rho_{0}(x)|\Psi(x)|\,dx<\infty$
there is a set $A_{1}$ of finite Lebesgue measure such that for every $x\in A_{1}$ we have that $\rho_{0}(x)<\min\{\frac{1}{n|\Psi(x)|},\frac{1}{nx^{2}}\}$. Similarly there is a set $A_{2}$ of finite Lebesgue measure such that for every $x\in A_{2}$ we have that $\rho_{1}(x)<\min\{\frac{1}{n|\Psi(x)|},\frac{1}{nx^{2}}\}$. We can even ask for $A_{2}$ to contain only Lebesgue points of $\rho_{1}$ to compensate for the lack of continuity.

Let $M_n>1$ with $M_n\in A_{1}\cap A_{2}$ such that
\begin{equation*}
  \int_{B^{c}(0,M_n)}\!\Big\lbrack\rho_{i}(x)+|\rho_{i}(x)\log\rho_{i}(x)|+\rho_{i}(x)x^{2}+\rho_{i}(x)|\Psi(x)|\Big\rbrack\, dx<\frac{1}{n},\hhh i=1,2 .
\end{equation*}

Let $\theta_{\epsilon}$ be as in Lemma~\ref{Ibound}. By the theory of mollifications there is a $\theta_{\epsilon(n)}$ that satisfies the following
\begin{itemize}
\item $\int_{B(0,M_n)}\big|(\rho^{M_n}_{1}*\theta_{\epsilon(n)})(x)-\rho_{1}(x)\big|\,dx < \frac{1}{n}$,
\item $\int_{B(0,M_n)}\big|((\rho^{M_n}_{1}*\theta_{\epsilon(n)})(x)-\rho_{1}(x))\Psi(x)\big|\,dx < \frac{1}{n}$,
\item $\Big|\int_{B(0,M_n)}(\rho^{M_n}_{1}*\theta_{\epsilon(n)})(x)\log(\rho^{M_n}_{1}*\theta_{\epsilon(n)})(x)-\rho_{1}(x)\log\rho_{1}(x)\,dx\Big| < \frac{1}{n}$,
\item $(\rho^{M_n}_{1}*\theta_{\epsilon(n)})(M_n) <\min\{\frac{1}{n\Psi(M_n)},\frac{1}{n}\}$,
\item $(\rho^{M_n}_{1}*\theta_{\epsilon(n)})(x)>0 ,\hhh \forall x\in B(0,M_n)$,
\end{itemize}
where $$\rho^{M_n}_{1}=\begin{cases}
                          \rho_{1}(x) &\mbox{if}\hhh |x|\leq M_n\\
                           0 &\mbox{if}\hhh |x|> M_n.
                         \end{cases}
$$

Since $\Psi(x)$ is continuous, there is a $0<a<1$ such that for $x\in [-M_n-a,-M_n+a]\cup[M_n-a,M_n+a]$ we have $\rho_{0}(M_n)<\min\{\frac{1}{n|\Psi(x)|},\frac{1}{nx^{2}}\}$ and
$(\rho^{M_n}_{1}*\theta_{\epsilon(n)})(M_n)<\min\{\frac{1}{n|\Psi(x)|},\frac{1}{nx^{2}}\}.$
Now define
\begin{align*}
  &g_{1,n}(x)=\begin{cases}
						(\rho^{M_n}_{1}*\theta_{\epsilon(n)})(x)	&\mbox{if}\hhh |x|\leq M_n,       \\
						(\rho^{M_n}_{1}*\theta_{\epsilon(n)})(M_n)(\frac{x-M_n+a}{a})^{2}   &\mbox{if}\hhh M_n<x< M_n+a,                  \\
						(\rho^{M_n}_{1}*\theta_{\epsilon(n)})(M_n)(\frac{x+M_n+a}{a})^{2}   &\mbox{if}\hhh -M_n-a<x<-M_n, \\
						0	&\mbox{if}\hhh |x|\geq M_n+a,
						\end{cases}
  \intertext{and}
  &g_{2,n}(x)=\begin{cases}
						0	                            &\mbox{if}\hhh |x|\leq M_n-a,      \\
	 					\rho_{0}(M_n)(\frac{M_n-x}{a})^{2}	&\mbox{if}\hhh M_n-a<x<M_n, \\
	 				    \rho_{0}(M_n)(\frac{M_n+a-x}{a})^{2}	        &\mbox{if}\hhh -M_n<x<-M_n+a, \\										                \rho_{0}(x)	                &\mbox{if}\hhh |x|\geq M_n. \\
						\end{cases}
\end{align*}

It is easy to check that  $\|\Delta g_{i,n}\|^{2}_{-1,g_{i,n}}$\footnote{This is a slight abuse of notation since $g_{i,n}$ are actually sub-probability measures.} and $\int g_{i,n}|\nabla\Psi|^{2}$ are finite for each $i=1,2$ and $n\in N$. Also,
\begin{multline*}
  \mathcal{S}(g_{1,n}) = \int_{B(0,M_n)}\left(g_{1,n}(x)\log g_{1,n}(x)-\rho_{1}(x)\log(\rho_{1}(x)\right)\,dx + \int_{B(0,M_n)}\rho_{1}(x)\log(\rho_{1}(x))\,dx  \\
+\int_{M_n<|x|<M_n+a} g_{1,n}(x)\log(g_{1,n}(x))\,dx\to \mathcal{S}(\rho_{1}) \hspace{1cm} \text{ as } n\to\infty.
\end{multline*}

Furthermore, by construction we have $\|g_{1,n}\|_{1}\rightarrow 1 , \|g_{2,n}\|_{1} \rightarrow 0 $ and we can finally define $k_{n}(x)$ by
\begin{equation}
k_{n}(x):= g_{1,n}(x)\frac{1-\|g_{2,n}\|_{1}}{\|g_{1,n}\|_{1}}+g_{2,n}(x),
\end{equation}
where $\|\cdot\|_{1}$ is the $L^1(\RR)$ norm. We have that $k_{n}$ is absolutely continuous and
\begin{align*}
\|\Delta k_{n}\|^{2}_{-1,k_{n}}=\int_{\RR}\frac{(k'_{n}(x))^{2}}{k_{n}(x)}\,dx&\leq \int_{\RR}\frac{\left((g_{1,n}(x)\frac{1-\|g_{2,n}\|_{1}}{\|g_{1,n}\|_{1}}+\int_{\RR}g_{2,n}(x))'\right)^{2}}{g_{1,n}(x)\frac{1-\|g_{2,n}\|_{1}}{\|g_{1,n}\|_{1}}+g_{2,n}(x)}\,dx\\
&\leq \int_{\RR}\frac{2\left( g'_{1,n}(x)\frac{1-\|g_{2,n}\|_{1}}{\|g_{1,n}\|_{1}}\right)^{2}}{g_{1,n}(x)\frac{1-\|g_{2,n}\|_{1}}{\|g_{1,n}\|_{1}}+g_{2,n}(x)}\,dx+\int_{\RR}\frac{2\left( g'_{2,n}(x)\right)^{2}}{g_{1,n}(x)\frac{1-\|g_{2,n}\|_{1}}{\|g_{1,n}\|_{1}}+g_{2,n}(x)}\,dx \\
&\leq \int_{\RR}\frac{1-\|g_{2,n}\|_{1}}{\|g_{1,n}\|_{1}}\frac{2(g'_{1,n}(x))^{2}}{g_{1,n}(x)}\,dx+\int_{\RR}\frac{2(g'_{2,n}(x))^{2}}{g_{2,n}(x)}\,dx\\ &\leq  \frac{2(1-\|g_{2,n}\|_{1})}{\|g_{1,n}\|_{1}} \|\Delta g_{1,n}\|^{2}_{-1,g_{1,n}} +2 \|\Delta g_{2,n}\|^{2}_{-1,g_{2,n}}.
\end{align*}
Hence  $\|\Delta k_{n}\|^{2}_{-1,k_{n}}<\infty$.
For the entropy functional we have:
\begin{align*}
|\mathcal{S}(k_{n})-\mathcal{S}(g_{1,n})|&=\Bigg{|}\int_{\RR}k_{n}(x)\log(k_{n}(x))\,dx- \int_{\RR}g_{1,n}(x)\log(g_{1,n}(x))\,dx\Bigg{|}\\
  &\leq \underbrace{\int_{B(0,M_n-a)}\!\Big| k_{n}(x)\log(k_{n}(x))-g_{1,n}(x)\log(g_{1,n}(x))\Big|\,dx}_{(I)}\\
  &\qquad+\underbrace{\int_{M_n-a\leq |x|\leq M_n+a} \!\Big|k_{n}(x)\log(k_{n}(x))-g_{1,n}(x)\log(g_{1,n}(x))\Big|\,dx}_{(II)}\\
  &\qquad+\underbrace{\int_{B^{c}(0,M_n+a)}\! \Big|k_{n}(x)\log(k_{n}(x))-g_{1,n}(x)\log(g_{1,n}(x))\Big|\,dx}_{(III)}.\\
\end{align*}
We now show that each of the three parts convergence to $0$ as $n\to\infty$. For the first part:
\begin{equation*}\begin{split}
(I) &=\int_{B(0,M_n-a)}\!\Big| g_{1,n}(x)\frac{1-\|g_{2,n}\|_{1}}{\|g_{1,n}\|_{1}}\log\Kl g_{1,n}(x)\frac{1-\|g_{2,n}\|_{1}}{\|g_{1,n}\|_{1}}\Kr-g_{1,n}(x)\log(g_{1,n}(x))\Big|\,dx\\
    &=\Big{|} 1-\frac{1-\|g_{2,n}\|_{1}}{\|g_{1,n}\|_{1}}\Big{|}\int_{B(0,M_n-a)}\Big|g_{1,n}(x)\log(g_{1,n}(x))\Big|\,dx\\
    &\hspace{6cm}+\frac{1-\|g_{2,n}\|_{1}}{\|g_{1,n}\|_{1}}\log{\frac{1-\|g_{2,n}\|_{1}}{\|g_{1,n}\|_{1}}}\int_{B(0,M_n-a)} |g_{1,n}(x)|\,dx \to 0.
\end{split}\end{equation*}

For the second part:
\begin{equation*}\begin{split}
(II)  &\leq \int_{M_n-a \leq |x| \leq M_n+a} \Big( |k_{n}(x)\log(k_{n}(x))|+|g_{1,n}(x)\log(g_{1,n}(x))|\Big)\,dx\\
      &=    \int_{M_n-a \leq |x| \leq M_n+a}  \big|(g_{1,n}(x)\frac{1-\|g_{2,n}\|_{1}}{\|g_{1,n}\|_{1}}+g_{2,n}(x))\log(g_{1,n}(x)\frac{1-\|g_{2,n}\|_{1}}{\|g_{1,n}\|_{1}}+g_{2,n}(x))\big|\,dx\\
      &\hspace{8cm} +\int_{M_n-a \leq |x| \leq M_n+a} |g_{1,n}(x)\log(g_{1,n}(x))|\,dx.
\end{split}\end{equation*}
Since $g_{1,n}(x), g_{2,n}(x)$ are smaller than $\frac{1}{n}$ in $M_n-a\leq |x|\leq M_n+a$, the right hand side converges to zero.

Part $(III)$ is smaller than $\frac{1}{n}$ by the first property of $M_n$ and therefore it converges to zero.

Finally
\begin{equation*}\begin{split}
\int_{\RR}|k_{n}(x)-\rho_{1}(x)||\Psi(x)|\,dx &=\int_{\RR}\big|g_{1,n}(x)\frac{1-\|g_{2,n}\|_{1}}{\|g_{1,n}\|_{1}}+g_{2,n}(x)-\rho_{1}(x)\big|\big|\Psi(x)\big|\,dx\\
                                              &\leq\Kl \frac{1-\|g_{2,n}\|_{1}}{\|g_{1,n}\|_{1}}\Kr \int_{\RR}|g_{1,n}(x)-\rho_{1}(x)||\Psi(x)|\,dx\\
                                              &\qquad +\Big{|} 1-\frac{1-\|g_{2,n}\|_{1}}{\|g_{1,n}\|_{1}}\Big{|} \int_{\RR}\!|\Psi(x)|\rho_{1}(x)\,dx +\int_{\RR}\!|\Psi(x)|g_{2,n}(x)\,dx\rightarrow 0
\end{split}\end{equation*}
Hence the second property of Proposition~\ref{th:dense recovery} is satisfied.
\end{proof}

\subsection*{Acknowledgements}
We would like to thank Nicolas Dirr, Mark Peletier and Johannes Zimmer for their initial suggestion and support during the project. The current proof of Lemma~\ref{lem:aux1} without probabilistic tools was done after a discussion with Mark Peletier. We also thank Jin Feng, Truyen Nguyen and Patrick van Meurs for their helpful discussion and comments. Manh Hong Duong has received funding from the ITN ``FIRST" of the Seventh Framework Programme of the European Community's (grant agreement number 238702).

\bibliographystyle{alpha}

\end{document}